\documentclass[a4paper]{amsart}
\usepackage{amsmath,amsthm,amssymb,latexsym,epic,bbm,comment,mathrsfs}
\usepackage{graphicx,enumerate,stmaryrd,xcolor}
\usepackage[all,2cell]{xy}
\xyoption{2cell}

\newtheorem{theorem}{Theorem}
\newtheorem{lemma}[theorem]{Lemma}
\newtheorem{corollary}[theorem]{Corollary}
\newtheorem{proposition}[theorem]{Proposition}

\usepackage[all]{xy}
\usepackage[active]{srcltx}
\usepackage[parfill]{parskip}
\usepackage{enumerate}

\newcommand{\ext}{\operatorname{ext}}
\newcommand{\Ext}{\operatorname{Ext}}
\newcommand{\soc}{\operatorname{soc}}

\newcommand{\inv}{^{-1}}

\newcommand{\cO}{\mathcal O}

\newcommand{\co}{\Delta_e/\Delta}

\begin{document}
\title[On first extensions in $\mathcal{S}$-subcategories of $\mathcal{O}$]
{On first extensions in $\mathcal{S}$-subcategories of $\mathcal{O}$}

\author[H.~Ko and V.~Mazorchuk]
{Hankyung Ko and Volodymyr Mazorchuk}

\begin{abstract}
We compute the first  extension group from a simple object to a proper 
standard object and, in some cases, the first  extension group from a  
simple object to a standard object in the principal block of 
an $\mathcal{S}$-subcategory of the BGG category $\mathcal{O}$
associated to a triangular decomposition of a semi-simple 
finite dimensional complex Lie algebra. 
\end{abstract}

\maketitle

\section{Introduction and description of the results}\label{s1}

Bernstein-Gelfand-Gelfand category $\mathcal{O}$ associated to a 
triangular decomposition of a semi-simple finite dimensional complex 
Lie algebra $\mathfrak{g}$ is about half a century old, it originates 
from the classical papers \cite{BGG1,BGG2}. However, it remains
an important and intensively studied object in modern representation 
theory, see \cite{Hu,CM1,CM2,KMM1,KMM2,KMM3,KMM4} for details. 
Category $\mathcal{O}$ has numerous analogues and generalizations, 
which include:
\begin{itemize}
\item parabolic category $\mathcal{O}$, see \cite{RC},
\item $\mathcal{S}$-subcategories in $\mathcal{O}$, see \cite{FKM,MS}.
\end{itemize}

Homological invariants of the above categories carry essential information
about both the structure and the properties of these categories. For category 
$\mathcal{O}$, many homological invariants are explicitly known, see
\cite{Ma1,Ma2,CM1,CM2,KMM1,KMM2,KMM3,KMM4} and references therein.
Using these results, many homological invariants for 
$\mathcal{S}$-subcategories in $\mathcal{O}$, for example, various projective
dimensions, can be computed using the approach of \cite[Section~4]{MPW}, 
especially using \cite[Theorem~15]{MPW}. In the present paper, inspired by the recent
results from \cite{KMM3,KMM4}, we take a closer look at the first extension
space between certain classes of structural object in $\mathcal{S}$-subcategories 
of $\mathcal{O}$.

We completely determine, in type $A$, 
the first extension space from a simple object to a proper
standard object in the regular block of an $\mathcal{S}$-subcategory of $\mathcal{O}$
in Theorems~\ref{fromw0inS} and \ref{thm4-21}. In many special cases (notably both for 
the dominant and the antidominant standard objects), we completely determine the 
first extension space from a simple object to a standard object in the regular 
block of an $\mathcal{S}$-subcategory of $\mathcal{O}$
in Proposition~\ref{prop-singular}. We also obtain some general results which reduce
the problem of computation of the first extension space from a simple object 
to a standard object in an $\mathcal{S}$-subcategory of $\mathcal{O}$
to a similar problem for certain objects in category $\mathcal{O}$, see
Proposition~\ref{prop65}.

The paper is organized as follows: Section~\ref{s2} contains preliminaries
on category $\mathcal{O}$ and its combinatorics. In Section~\ref{s3}
we survey some of the recent results of \cite{KMM3,KMM4} which describe extensions
from a simple highest weight module to a Verma module in category $\mathcal{O}$.
In Section~\ref{s4} we recall the definition and basic properties of 
$\mathcal{S}$-subcategories in $\mathcal{O}$. Section~\ref{s9}
is devoted to explicit description of the first extensions space from
a simple to a proper standard object in $\mathcal{S}$-subcategories in $\mathcal{O}$
in type $A$. We also formulate a number of general results which hold in all types.
In Section~\ref{s6} we similarly look at the first extensions space from
a simple to a standard object. 
We complete the paper with some examples
in Section~\ref{s5}. This includes a detailed $\mathfrak{sl}_3$-example
(for a rank one parabolic) as well as various examples of non-trivial 
extension from a simple to a proper standard object for the algebra $\mathfrak{sl}_4$.

\vspace{2mm}

\subsection*{Acknowledgments}
For the second author, this research is supported by the Swedish
Research Council.

\section{Preliminaries on category $\mathcal{O}$}\label{s2}

\subsection{Category $\mathcal{O}$}\label{s2.1}

Let $\mathfrak{g}$ be a semi-simple finite dimensional complex Lie algebra
with a fixed triangular decomposition 
$\mathfrak{g}=\mathfrak{n}_-\oplus \mathfrak{h}\oplus \mathfrak{n}_+$,
see \cite{Hu,MP} for details. Associated to this datum, we have
the Bernstein-Gelfand-Gelfand category $\mathcal{O}$ defined as the
full subcategory of the category of all finitely generated
$\mathfrak{g}$-modules, consisting of all $\mathfrak{h}$-\-di\-a\-go\-na\-liz\-a\-ble
and locally $U(\mathfrak{n}_+)$-finite modules, cf. \cite{BGG1,BGG2,MP,Hu}.

Simple modules in $\mathcal{O}$ are exactly the simple highest weight modules
$L(\lambda)$, where $\lambda\in \mathfrak{h}^*$,
see  \cite[Chapter~7]{Di} for details. For each such $\lambda$,
we also have in $\mathcal{O}$ the corresponding
\begin{itemize}
\item Verma module $\Delta(\lambda)$,
\item dual Verma module $\nabla(\lambda)$,
\item indecomposable projective module $P(\lambda)$,
\item indecomposable injective module $I(\lambda)$,
\item indecomposable tilting module $T(\lambda)$.
\end{itemize}

Consider the principal block $\mathcal{O}_0$ of $\mathcal{O}$, which is 
defined as the indecomposable direct summand containing the trivial
$\mathfrak{g}$-module $L(0)$. Simple modules in $\mathcal{O}_0$ are
indexed by the elements of the Weyl group $W$ of $\mathfrak{g}$. 
For $w\in W$, we have the corresponding simple module $L_w:=L(w\cdot 0)$,
where $w\cdot{}_-$ denotes the usual dot-action of the Weyl group on
$\mathfrak{h}^*$. We will similarly denote by $\Delta_w$, $\nabla_w$, 
$P_w$, $I_w$ and $T_w$ the other structural modules corresponding to $L_w$.

We will use $\mathrm{Ext}$ and $\mathrm{Hom}$ to denote extensions and
homomorphisms in $\mathcal{O}$, respectively. The simple preserving duality
on $\mathcal{O}$ is denoted by $\star$.

\subsection{Graded category $\mathcal{O}$}\label{s2.2}
The category $\mathcal{O}_0$ admits a $\mathbb{Z}$-graded lift 
$\mathcal{O}_0^{\mathbb{Z}}$, see \cite{So2,St}. All structural modules
in $\mathcal{O}_0$ admit graded lifts (unique up to isomorphism and 
shift of grading). We will use the same notation as for ungraded modules
to denote the following graded lifts of the
structural modules in $\mathcal{O}_0^{\mathbb{Z}}$:
\begin{itemize}
\item $L_w$ denotes the graded simple object concentrated in degree $0$,
\item ${\Delta}_w$ denotes the graded Verma module  with top in degree $0$,
\item $\nabla_w$ denotes the graded dual Verma module with socle
in degree $0$,
\item $P_w$ is the graded indecomposable projective module with top
in degree $0$,
\item $I_w$ is the graded indecomposable injective modules with socle
in degree $0$,
\item $T_w$ is the graded indecomposable tilting module having the unique
$L_w$ subquotient in degree $0$.
\end{itemize}

For $k\in\mathbb{Z}$, we denote by $\langle k\rangle$ the functor which shifts 
the grading by $k$,
with the convention that $\langle 1\rangle$ maps degree $0$ to degree $-1$.
We will use $\mathrm{ext}$ and $\mathrm{hom}$ to denote extensions and
homomorphisms in $\mathcal{O}_0^{\mathbb{Z}}$, respectively. 
Note that, for any $k\geq 0$ and any two structural modules $M$ and 
$N$ with fixed graded lifts $\mathtt{M}$ and $\mathtt{N}$, we have
\begin{displaymath}
\mathrm{Ext}^{k}(M,N)\cong
\bigoplus_{i\in\mathbb{Z}}
\mathrm{ext}^{k}(\mathtt{M},\mathtt{N}\langle i\rangle).
\end{displaymath}
%The graded version of $\star$ is denoted by $\text{\textcircled{$\star$}}$.

\subsection{Combinatorics of category $\mathcal{O}_0^{\mathbb{Z}}$}\label{s2.3}

Let $\mathbf{H}$ denote the Hecke algebra of $W$ over $\mathbb{Z}[v,v^{-1}]$
in the normalization of \cite{So2}. It has the standard basis
$\{H_w\,:\,w\in W\}$ and the Kazhdan-Lusztig basis $\{\underline{H}_w\,:\,w\in W\}$.
The Kazhdan-Lusztig polynomials $\{p_{x,y}\,:\,x,y\in W\}$ are the entries of 
the transformation matrix between these two bases, that is 
\begin{displaymath}
\underline{H}_y=\sum_{x\in W}p_{x,y}H_x, \text{ for all } y\in W.
\end{displaymath}

Taking the Grothendieck group gives rise to an isomorphism of 
$\mathbb{Z}[v,v^{-1}]$-modules as follows:
\begin{displaymath}
\mathrm{Gr}(\mathcal{O}_0^{\mathbb{Z}}) \cong \mathbf{H},\qquad
[{\Delta}_w]\mapsto H_w, \text{ for } w\in W.
\end{displaymath}
Here the $\mathbb{Z}[v,v^{-1}]$-module structure on 
$\mathrm{Gr}(\mathcal{O}_0^{\mathbb{Z}})$ is given be letting
the element $v$ act as $\langle -1\rangle$.
This isomorphism maps $\mathtt{P}_w$ to $\underline{H}_w$, for $w\in W$.

\subsection{Kazhdan-Lusztig orders and cells}\label{s2.4}

Following \cite{KL}, for $x,y\in W$, we write $x\geq_L y$ provided that
there is $w\in W$ such that $\underline{H}_x$ appears with a non-zero
coefficient in $\underline{H}_w\underline{H}_y$. This defines
the {\em left pre-order} on $W$. The equivalence classes with respect to
this pre-order are called {\em left cells} and the corresponding equivalence
relation is denoted $\sim_L$.

Similarly, for $x,y\in W$, we write $x\geq_R y$ provided that
there is $w\in W$ such that $\underline{H}_x$ appears with a non-zero
coefficient in $\underline{H}_y\underline{H}_w$. This defines
the {\em right pre-order} on $W$. The equivalence classes with respect to
this pre-order are called {\em right cells} and the corresponding equivalence
relation is denoted $\sim_R$.

Finally, for $x,y\in W$, we write $x\geq_J y$ provided that
there are $w,w'\in W$ such that $\underline{H}_x$ appears with a non-zero
coefficient in $\underline{H}_w\underline{H}_y\underline{H}_{w'}$. This defines
the {\em two-sided pre-order} on $W$. The equivalence classes with respect to
this pre-order are called {\em two-sided cells} and the corresponding equivalence
relation is denoted $\sim_J$. 

The two-sided pre-order induces a partial order on the set of the two-sided cells.
The maps $w\mapsto w_0w$ and $w\mapsto ww_0$
induce anti-involution on the poset of two-sided cells, see \cite[Chapter~6]{BB}. 
In particular, the poset of two-sided cells has the minimum element $\{e\}$ and 
the maximum element $\{w_0\}$. In type $A_1$, there is nothing else.
Outside type $A_1$, removing these two extreme cells, we again get a 
poset with the minimum and the maximum element. The new minimum element 
is the cell containing all simple reflections, called the 
{\em small cell} (see \cite{KMMZ}), while the new maximum element is the image of 
the small cell under the $w\mapsto w_0w$ anti-involution (note that the
two new extreme cells coincide in rank $2$). We call
this new maximum cell the {\em penultimate cell} and denote it
by $\mathcal{J}$.

% \subsection{Endofunctors of $\mathcal{O}_0$}\label{s2.5}
% 
% The category $\mathcal{O}_0$ is equipped with the action of various families
% of functors, see \cite{BG,Jo,Ca,AS,KM,MS,MS2,Hu} and references therein. 
% For $w\in W$, we denote by
% \begin{itemize}
% \item $\theta_w$ the indecomposable projective endofunctor of $\mathcal{O}_0$
% sending $\Delta_e$ to $\Delta_w$, see \cite{BG};
% \item $\mathrm{C}_w$ the corresponding shuffling functors, see \cite{Ca,MS};
% \item $\mathrm{K}_{w^{-1}}$ the right adjoint of $\mathrm{C}_w$, see \cite{MS2};
% \item $\mathrm{T}_w$ the corresponding twisting functors, see \cite{AS,KM};
% \item $\mathrm{J}_{w^{-1}}$ the right adjoint of $\mathrm{T}_w$
% (a.k.a. Joseph's completion functor), see \cite{Jo,AS,KM}.
% \end{itemize}
% We will only use some of these functors in our arguments below,
% however, we present here the full list with some related references for
% the convenience of the reader. 
% 
% \com{this section is not necessary. I wrote a shorter proof for the only place T appeared. So it is only Theta that appears, but it is explained when it appears.}

\section{First extension from a simple to a Verma module
in category $\mathcal{O}$}\label{s3}

In this section, we briefly summarize the results from \cite{Ma1,KMM3,KMM4}
which describe the first extension from a simple module to a Verma module.

\subsection{First extension to a Verma from the anti-dominant simple}\label{s3.1}

We have the usual length function $\ell$ on the Weyl group $W$ considered as a
Coxeter group with respect to the simple reflections determined by our fixed 
triangular decomposition of $\mathfrak{g}$. For $w\in W$, the value
$\ell(w)$ is the length of a reduced expression of $w$. We also have the
content function $\mathbf{c}:W\to\mathbb{Z}_{\geq 0}$. For $w\in W$, 
the value  $\mathbf{c}(w)$ is the number of {\em different} simple 
reflections which appear in a reduced expression of $w$ (from the Coxeter
relations it follows that this number does not depend on the choice of a
reduced expression). 

As usual, we denote by $w_0$ the longest element of $W$.
The following result is \cite[Theorem~32]{Ma1}.

\begin{theorem}\label{thm1}
For $w\in W$ and $i\in\mathbb{Z}$, we have 
\begin{displaymath}
\dim\mathrm{ext}^1(\Delta_{w_0},\Delta_w\langle i\rangle)=
\begin{cases}
\mathbf{c}(w_0w),& i=\ell(w_0)-\ell(w)-2;\\
0,& \text{otherwise}.
\end{cases}
\end{displaymath}
\end{theorem}

\subsection{First extension to a Verma module from other simple modules and inclusion of Verma modules}\label{ss:socandext}

Recall the following properties of (graded) Verma modules:
\begin{itemize}
\item every non-zero map between two Verma modules is injective;
\item for $x,y\in W$, we have $\mathrm{hom}_{}(\Delta_x,\Delta_y\langle d \rangle)\neq 0$
if and only if $x\geq y$ in the Bruhat order and $d=\ell(x)-\ell(y)$;
\item $\dim \mathrm{Hom}_{}(\Delta_x,\Delta_y)\leq 1$, for all $x,y\in W$.
\end{itemize}
The ungraded versions of these properties can be found in \cite[Chapter~7]{Di}.
The graded version of the second property follows by matching the degrees 
using standard arguments, see e.g. \cite{St}.

In particular, each $\Delta_x\langle -\ell(x)\rangle$ injects to $\Delta_e$, 
and the cokernel $\Delta_e/(\Delta_x\langle -\ell(x)\rangle)$ belongs to 
$\cO_0^\mathbb{Z}$. To ease the notation, we denote the latter by $\co_x$.
These cokernels control the first extension from non-anti-dominant simple 
modules to Verma modules in the following way, as observed in \cite{KMM3}. 

\begin{proposition}\label{extandsoc}
For each $x,w\in W$, with $x\neq w_0$, we have
\begin{displaymath}
\dim\ext^1(L_x\langle d\rangle,\Delta_w\langle -\ell(w) \rangle) 
= [\soc\co_w:L_x\langle d \rangle]. 
\end{displaymath}
\end{proposition}

The proof is similar to the second part of \cite[Proof of Corollary 2]{KMM3}. 
A similar argument will also be given in Proposition~\ref{extandsocinS}.

The rest of this section describes the cokernels $\co_w$, for $w\in W$. 
To do this, we need to dive into poset-theoretic properties of the Bruhat order. 
An element $w\in W$ is called {\em join-irreducible} provided that it is not a 
join (supremum) of other elements, that is, there is no $U\subset W$ 
with $w\not\in U$ such that $w = \bigvee U$. The set of all join-irreducible 
elements, denoted by $\mathbf{B}$, is called the \emph{base} of the poset $W$.

\subsection{Cokernel of inclusion between Verma modules in type $A$}\label{s3.2}

In a few coming subsections we restrict to the case of type $A$.
The join-irreducible elements in $W$ of type $A$ are explicitly
identified in  \cite{LS} as the bigrassmannian elements. 
An element $w\in W$ is called {\em bigrassmannian} provided that there is
a unique simple reflection $s$ such that $\ell(sw)<\ell(w)$ and there is
a unique simple reflection $t$ such that $\ell(w)<\ell(wt)$. 
In type $A$, the base $\mathbf{B}$ agrees 
with the set of bigrassmannian elements in $W$.

The Kazhdan-Lusztig two-sided order is also easier in type $A$. 
The classical Robinson-Schensted correspondence
\begin{displaymath}
\mathrm{RS}:S_n\longrightarrow \coprod_{\lambda\vdash n}
\mathbf{SYT}_\lambda\times \mathbf{SYT}_\lambda
\end{displaymath}
assigns to $w\in W$ a pair $\mathrm{RS}(w)=(p_w,q_w)$ of 
standard Young tableaux of shape $\lambda=:\mathrm{sh}(w)$, where 
$\lambda$ is a partition of $n$, see \cite[Section~3.1]{Sa}. 
By \cite[Theorem~1.4]{KL}, we have
\begin{itemize}
\item $x\sim_L y$ if and only if $q_x=q_y$;
\item $x\sim_R y$ if and only if $p_x=p_y$;
\item $x\sim_J y$ if and only if $\mathrm{sh}(x)=\mathrm{sh}(y)$.
\end{itemize}
The poset of all two-sided cells with respect to the two-sided order 
is isomorphic to the poset of all partitions of $n$ with respect to
the dominance order, see \cite{Ge}. 

Recall that $\mathcal J$ denotes the penultimate cell with respect to the two-sided order.
In type $A$, the elements in $\mathcal{J}$ are naturally indexed by pairs of
simple reflections in $W$: for any pair $(s,t)$ of simple reflections in
$W$, there is a unique element $w_{s,t}\in \mathcal{J}$ such that 
$w_0=sw_{s,t}=w_{s,t}t$.

We now formulate the main result of \cite{KMM3}, that is
\cite[Theorem~1]{KMM3}.
\begin{theorem}\label{thm2}
{\hspace{1mm}}

\begin{enumerate}[$($i$)$]
\item\label{thm2.1} For $w\in S_n$, the module $\Delta_e/\Delta_w$
has simple socle if and only if $w\in \mathbf{B}$.
\item\label{thm2.2} The map $\mathbf{B}\ni w\mapsto \mathrm{soc}(\Delta_e/\Delta_w)$
induces a bijection between $\mathbf{B}$ and simple subquotients of
$\Delta_e$ of the form $L_x$, where $x\in\mathcal{J}$.
\item\label{thm2.3} For $w\in S_n$, the simple subquotients of 
$\Delta_e/\Delta_w$ of the form $L_x$, where $x\in\mathcal{J}$, correspond, under the
bijection from \eqref{thm2.2}, to $y\in \mathbf{B}$ such that $y\leq w$.
\item\label{thm2.4} For $w\in S_n$, the socle of $\Delta_e/\Delta_w$ 
consists of all $L_x$, where $x\in\mathcal{J}$, which correspond, under the
bijection from \eqref{thm2.2}, to the Bruhat maximal elements in
the set $\{y\in \mathbf{B}\,:\, y\leq w\}$.
\end{enumerate}
\end{theorem}

Motivated by the last claim, we denote 
$\mathbf{BM}(w):=\max \{y\in \mathbf{B}\,:\, y\leq w\}$.

The socle of the cokernel of an inclusion between two arbitrary Verma modules
can be described using Theorem~\ref{thm2}. The following corollary
is \cite[Corollary~23]{KMM3}.

\begin{corollary}\label{corthm2}
Let $v,w\in S_n$ be such that $v<w$.
\begin{enumerate}[$($i$)$]
\item\label{corthm2.1} The bijection from Theorem~\ref{thm2}\eqref{thm2.2} induces
a bijection between simple subquotients of 
$\Delta_v/\Delta_w$ of the form $L_x$, where $x\in\mathcal{J}$, 
and $y\in \mathbf{B}$ such that $y\leq w$ and $y\not\leq v$.
\item\label{corthm2.2} The socle of $\Delta_v/\Delta_w$ 
consists of all $L_x$, where $x$ corresponds to an element in $\mathbf{BM}(w)\setminus \mathbf{BM}(v)$.
\end{enumerate}
\end{corollary}

\subsection{First extension to a Verma from other simples in type $A$}\label{s3.3}

Let $w\in \mathbf{B}$ be such that $\ell(sw)<\ell(w)$ and $\ell(wt)<\ell(w)$,
for two simple reflections $s$ and $t$. Denote by $\Phi: \mathbf{B}\to \mathcal{J}$ 
the map which sends such $w$ to $w_{s,t}$. Theorem~\ref{thm2} and Proposition~\ref{extandsoc} has the following 
consequence:

\begin{corollary}\label{cor2}
Let $x,y\in S_n$ with $x\neq w_0$. Then we have
\begin{displaymath}
\dim\mathrm{Ext}_{}^{1}(L_x,\Delta_y)=
\dim\mathrm{Ext}_{}^{1}(\nabla_y,L_x)=
\begin{cases}
1, & x\in \Phi(\mathbf{BM}(y));\\
0, & \text{otherwise}. 
\end{cases}
\end{displaymath}
\end{corollary}

\subsection{Extensions in singular blocks in type $A$}\label{s3.35}

Let $\lambda$ be a dominant integral weight and $\mathcal{O}_\lambda$
the indecomposable summand of $\mathcal{O}$ containing $\lambda$.
If $\lambda$ is regular, then $\mathcal{O}_\lambda$ is equivalent to
$\mathcal{O}_0$. In the general case, denote by $W^\lambda$ the 
stabilizer of $\lambda$ with respect to the dot action of $W$.
Simple objects in $\mathcal{O}_\lambda$ are then in a natural bijection
with the cosets in $W/{}_{W^\lambda}$.

For $w\in W$ denote by $\overline{w}$ the unique longest element
in $wW^\lambda$. Also, denote by $\underline{w}$ the unique shortest 
element in $wW^\lambda$. The following claim is \cite[Theorem~16]{KMM3}:

\begin{theorem}\label{thm75}
Let $x,y\in S_n$ and let $\mu$ be an integral, dominant weight. Then we have
\begin{displaymath}
\dim\mathrm{Ext}_{\cO}^{1}(L(x\cdot \lambda),\Delta(y\cdot \lambda))=
\begin{cases}
\mathbf{c}(\overline{x}\underline{y})-\mathrm{rank}(W^\lambda), & \overline{x}=w_0;\\
1, & \overline{x}\in \Phi(\mathbf{BM}(\underline{y}));\\
0, & \text{otherwise}. 
\end{cases}
\end{displaymath}
\end{theorem}

\subsection{The graded picture in type $A$}\label{s3.4}

Corollary~\ref{cor2} admits a graded lift. Let $s_1,\dots,s_{n-1}$
be the simple reflections in $S_n$ such that the corresponding Dynkin diagram is 
\begin{displaymath}
\xymatrix{s_1\ar@{-}[rr]&&s_2\ar@{-}[rr]&&\dots\ar@{-}[rr]&&s_{n-1}} 
\end{displaymath}
For $i,j\in\{1,2,\dots,n-1\}$, let
\begin{displaymath}
{}_i\mathbf{B}_j:=
\{w\in \mathbf{B}\,:\,\ell(s_iw)<\ell(w)\text{ and }\ell(ws_j)<\ell(w)\}. 
\end{displaymath}
The set ${}_i\mathbf{B}_j$ consists of $\min\{i,j,n-i,n-j\}$ elements
which can be described very explicitly, see \cite[Subsection~4.2]{KMM3}.
For example, here are the three elements of ${}_4\mathbf{B}_3$ in $S_7$
and their graphs:

\resizebox{\textwidth}{!}{
$
\left(\begin{array}{ccccccc}1&2&3&4&5&6&7\\1&2&5&3&4&6&7\end{array}\right),\quad
\left(\begin{array}{ccccccc}1&2&3&4&5&6&7\\1&5&6&2&3&4&7\end{array}\right),\quad
\left(\begin{array}{ccccccc}1&2&3&4&5&6&7\\5&6&7&1&2&3&4\end{array}\right),
$}

\resizebox{\textwidth}{!}{
$
\xymatrix@C=1.5mm@R=8mm{1\ar@{-}[d]&2\ar@{-}[d]&3\ar@{-}[drr]&4\ar@{-}[dl]
&5\ar@{-}[dl]&6\ar@{-}[d]&7\ar@{-}[d]\\1&2&3&4&5&6&7},\qquad
\xymatrix@C=1.5mm@R=8mm{1\ar@{-}[d]&2\ar@{-}[drrr]&3\ar@{-}[drrr]&4\ar@{-}[dll]
&5\ar@{-}[dll]&6\ar@{-}[dll]&7\ar@{-}[d]\\1&2&3&4&5&6&7},\qquad
\xymatrix@C=1.5mm@R=8mm{1\ar@{-}[drrrr]&2\ar@{-}[drrrr]&3\ar@{-}[drrrr]&4\ar@{-}[dlll]
&5\ar@{-}[dlll]&6\ar@{-}[dlll]&7\ar@{-}[dlll]\\1&2&3&4&5&6&7}.
$}

The elements of ${}_i\mathbf{B}_j$ form naturally a chain with respect to the
Bruhat order on $S_n$. This allows us to index the elements of the set 
${}_i\mathbf{B}_j$ via the tuples $(i,j,k)$, where $0\leq k\leq \min\{i,j,n-i,n-j\}-1$, 
increasingly along the Bruhat order. 
Since Theorem~\ref{thm2} is gradable, see \cite[Proposition~22]{KMM3}, we can lift Corollary~\ref{cor2} to the
graded setup by Proposition~\ref{extandsoc}.

\begin{proposition}\label{cor2graded}
Let $y\in S_n$ and $x=\Phi((i,j,k))$, where $(i,j,k)\in \mathbf{BM}(y)$.
Then the unique degree $m\in\mathbb{Z}$ for which
$\dim\mathrm{ext}_{}^{1}(L_x\langle -m \rangle,\Delta_y\langle -\ell(y)\rangle)=1$ is 
\[m = \frac{(n-1)(n-2)}{2}+\vert i-j\vert+2k.
\]
\end{proposition}

Similarly, Theorem~\ref{thm75} can also be graded.

\begin{proposition}\label{cor2grade2}
Let $y=\underline{y}\in S_n$ and $x=\overline{x}=\Phi((i,j,k))$, where $(i,j,k)\in \mathbf{BM}(y)$.
Then the unique degree $m\in\mathbb{Z}$ for which
$\dim\mathrm{ext}_{}^{1}(L(x\cdot \lambda)\langle -m \rangle,\Delta(y\cdot \lambda)\langle -\ell(y) \rangle)=1$ is 
\[m = \frac{(n-1)(n-2)}{2}+\vert i-j\vert+2k.\]
\end{proposition}

\subsection{First extension to a Verma from other simples in other types}\label{s3.6}

By Proposition~\ref{extandsoc}, the problem again reduces to determining 
the (socles of) $\co_w$, for $w\in W$. However, the latter does not seem to follow 
a uniformly describable pattern, in general. In particular, it is shown in 
\cite{KMM4} that none of the statements in Theorem~\ref{thm2} is true, in general, 
in other types. 

What remains to be true is that, for $x,w\in W$ with $x\neq w_0$,
we have $\Ext^1(L_x,\Delta_w) = 0$, unless $x\in \mathcal J$.
Another partial result is an upper bound.
Let
\[_s \mathbf{BM}_t (w)= \{z\in \mathbf{BM}(w)\ |\ \text{ $sz<z$ and $zt<z$}\}, \]
for $w\in W$ and  $s,t$ simple reflections. The following is
Theorem~F(c) in \cite{KMM4}.

\begin{theorem}\label{thmnotA}
Let $w\in W$ and $x\in \mathcal J$. If $s,t$ are simple 
reflections in $W$ such that $sx>x$ and $xt>x$, then
\begin{equation}\label{eqc}
\dim\ext^1(L_x,\Delta_w\langle d\rangle) \leq \dim\ext^1(L_x,\Delta_b\langle d \rangle), 
\end{equation}
for all $d\in \mathbb Z$, where $b$ is the join of $\!_s\mathbf{BM}_t(w)$. The right hand side of \eqref{eqc} is again bounded by
\begin{equation}
\dim \Ext^1(L_x,\Delta_b)\leq |_s\mathbf{BM}_t(w) |.
\end{equation}
In particular, $\Ext^1(L_x,\Delta_w)=0$, if $_s\mathbf{BM}_t (w)=\emptyset$.
\end{theorem}

Using further computation, it is determined in \cite{KMM4} that
\begin{displaymath}
\dim\ext^1(L_x,\Delta_b\langle d \rangle)
\end{displaymath}
is bounded by $1$ in type $B$, by $2$ in types $DF$, and by $3$, $4$, and $6$ 
in types $E_6$, $E_7$ and $E_8$, respectively.

The paper \cite{KMM4} develops several techniques to compute 
specific $\co_w$. Thus for a given $w\in W$, it is often possible 
to determine $\ext^1(L_x,\Delta_w\langle d\rangle)$, for all $x\in W$ 
and $d\in \mathbb Z$. See \cite[Section~5]{KMM4} for details.

\section{$\mathcal{S}$-subcategories in $\mathcal{O}$}\label{s4}

In this subsection we recall the definition and basic properties of 
$\mathcal{S}$-subcategories in $\mathcal{O}$ from \cite{FKM,MS}.

\subsection{Definition}\label{s4.1}

Let $\mathfrak{p}$ be a parabolic subalgebra of $\mathfrak{g}$
containing $\mathfrak{h}\oplus\mathfrak{n}_+$. We denote by
$W^{\mathfrak{p}}$ the corresponding parabolic subgroup of $W$
and by $w_0^\mathfrak{p}$ the longest element of $W^{\mathfrak{p}}$.
Denote by $\mathtt{X}_\mathfrak{p}^{\mathrm{long}}$ 
and $\mathtt{X}_\mathfrak{p}^{\mathrm{short}}$ 
the sets of the longest and the shortest representatives in the 
$W^{\mathfrak{p}}$-cosets from 
${}_{W^{\mathfrak{p}}}\hspace{-1mm}\setminus\hspace{-1mm} W$, respectively.
The map $w_0^\mathfrak{p}\cdot{}_-: \mathtt{X}_\mathfrak{p}^{\mathrm{long}}
\to \mathtt{X}_\mathfrak{p}^{\mathrm{short}}$ is a bijection with inverse
$w_0^\mathfrak{p}\cdot{}_-: \mathtt{X}_\mathfrak{p}^{\mathrm{short}}
\to \mathtt{X}_\mathfrak{p}^{\mathrm{long}}$.

Recall that the parabolic category $\mathcal{O}^{\mathfrak{p}}_0$
is defined in \cite{RC} as the Serre subcategory of $\mathcal{O}_0$
generated by all $L_w$, where $w\in \mathtt{X}_\mathfrak{p}^{\mathrm{short}}$.

We define the $\mathcal{S}$-subcategory $\mathcal{S}_0^{\mathfrak{p}}$ of 
$\mathcal{O}_0$ as the quotient of $\mathcal{O}_0$ modulo the 
Serre subcategory $\mathcal{Q}_\mathfrak{p}$ generated by all 
$L_w$, where $w\not\in \mathtt{X}_\mathfrak{p}^{\mathrm{long}}$.
We denote by $\pi_{\mathfrak{p}}:\mathcal{O}_0\to \mathcal{S}_0$ the Serre
quotient functor.

The category $\mathcal{S}_0^{\mathfrak{p}}$ admits various realizations as 
a full subcategory of $\mathcal{O}_0$. For example, $\mathcal{S}_0^{\mathfrak{p}}$
is equivalent to the full subcategory of $\mathcal{O}_0$ consisting
of all $M$ which have a projective presentation of the form
\begin{displaymath}
X\to Y\to M\to 0, 
\end{displaymath}
such that, for each $P_w$ appearing as a summand of $X$ or $Y$, 
we have $w\in \mathtt{X}_\mathfrak{p}^{\mathrm{long}}$.
Alternatively, $\mathcal{S}_0^{\mathfrak{p}}$ is equivalent to the full 
subcategory of $\mathcal{O}_0$ consisting of all $N$ which have 
an injective copresentation of the form
\begin{displaymath}
0\to N\to X\to Y
\end{displaymath}
such that, for each $I_w$ appearing as a summand of $X$ or $Y$, 
we have $w\in \mathtt{X}_\mathfrak{p}^{\mathrm{long}}$.
By abstract nonsense, see \cite{Au}, $\mathcal{S}_0^{\mathfrak{p}}$ is equivalent to 
the module category over the endomorphism algebra $A^\mathfrak{p}$
of the direct sum of all $P_w$, where $w\in \mathtt{X}_\mathfrak{p}^{\mathrm{long}}$.

We also note that, in the case $W^\mathfrak{p}$ is of type $A_1$, the category
$\mathcal{S}_0^{\mathfrak{p}}$ is the Serre quotient of $\mathcal{O}_0$ by 
$\mathcal{O}^{\mathfrak{p}}_0$. In this case
$W=\mathtt{X}_\mathfrak{p}^{\mathrm{long}}\bigcup\mathtt{X}_\mathfrak{p}^{\mathrm{short}}$.

The graded version $(\mathcal{S}_0^{\mathfrak{p}})^{\mathbb Z}$
of $\mathcal{S}_0^{\mathfrak{p}}$ is similarly defined as the Serre quotient of
$\cO_0^{\mathbb Z}$ by the Serre subcategory of the latter category
generated by all $L_w\langle i\rangle$, where 
$w\not\in \mathtt{X}_\mathfrak{p}^{\mathrm{long}}$ and $i\in\mathbb{Z}$.
We use the same notation $\pi_{\mathfrak p}$ for the graded Serre quotient functor.
The above alternative descriptions have the obvious graded analogues.
For example, $(\mathcal{S}_0^{\mathfrak{p}})^{\mathbb Z}$ is equivalent to
the full subcategory of $\cO_0^{\mathbb Z}$ consisting of all objects which have 
a projective presentation as above with indecomposable summands of the form
$P_w\langle i\rangle$, where $w\in \mathtt{X}_\mathfrak{p}^{\mathrm{long}}$
and $i\in\mathbb{Z}$. Similarly for the injective copresentation.

\subsection{Origins and motivation}\label{s4.2}

$\mathcal{S}$-subcategories in $\mathcal{O}$ were formally defined in \cite{FKM}.
They provide a uniform description for a number of generalizations of category 
$\mathcal{O}$ in \cite{FKM1,FKM2,FKM4,Ma3,MiSo}. Notably, these include
various categories of Gelfand-Zeitlin module, see \cite{Ma3}, and 
Whittaker modules, see \cite{MiSo}.

The realization of the $\mathcal{S}$-subcategories in $\mathcal{O}$ 
as projectively presentable modules in $\mathcal{O}$
was studied in \cite{MS}. In particular, in \cite{MS} it was shown that the
action of projective functors on $\mathcal{S}_0$ categorifies
the permutation $W$-module for $W^{\mathfrak{p}}$, i.e., the
$W$-module obtained by inducing the trivial $W^{\mathfrak{p}}$-module
up to $W$ (see also \cite{MS3} for further details).

\subsection{Stratified structure}\label{s4.3}
Here we recall some structural properties of $\mathcal{S}_0^{\mathfrak{p}}$
established in \cite{FKM,MS}.

For $w\in \mathtt{X}_\mathfrak{p}^{\mathrm{long}}$, denote by 
\begin{itemize}
\item $L^\mathfrak{p}_w$ the object $\pi_\mathfrak{p}(L_w)$
in $\mathcal{S}_0^{\mathfrak{p}}$;
\item $P^\mathfrak{p}_w$ the object $\pi_\mathfrak{p}(P_w)$
in $\mathcal{S}_0^{\mathfrak{p}}$;
\item $I^\mathfrak{p}_w$ the object $\pi_\mathfrak{p}(I_w)$
in $\mathcal{S}_0^{\mathfrak{p}}$;
\item $T^\mathfrak{p}_w$ the object 
$\pi_\mathfrak{p}(T_{w_0^\mathfrak{p}w}\langle -\ell(w_0^\mathfrak{p})\rangle)$
in $\mathcal{S}_0^{\mathfrak{p}}$.
\end{itemize}
By construction, $L^\mathfrak{p}_w$ is simple
and $\{L^\mathfrak{p}_w\,:\, w\in \mathtt{X}_\mathfrak{p}^{\mathrm{long}}\}$
is a complete and irredundant list of representatives of 
simple objects in $\mathcal{S}_0^\mathfrak{p}$. The objects $P^\mathfrak{p}_w$ and 
$I^\mathfrak{p}_w$ are the corresponding indecomposble 
projectives and injectives in $\mathcal{S}_0^\mathfrak{p}$, 
respectively. For structural modules, we will use the same notation for 
the ungraded versions of the modules  and for their graded versions.
The latter are obtained by applying the graded version of $\pi_\mathfrak{g}$
to the standard graded lifts of structural modules.

For $w\in \mathtt{X}_\mathfrak{p}^{\mathrm{long}}$, denote by 
$\overline{\Delta}^\mathfrak{p}_w$ the object $\pi_\mathfrak{p}(\Delta_w)$
in $\mathcal{S}_0^\mathfrak{p}$. 
Then $\overline{\Delta}^\mathfrak{p}_w\cong
\pi_\mathfrak{p}(\Delta_{xw} \langle \ell(x)\rangle)$, for all $x\in W^{\mathfrak{p}}$.
The object $\overline{\Delta}^\mathfrak{p}_w$ is called the
{\em proper standard object} corresponding to the element $w$.

Further, for $w\in \mathtt{X}_\mathfrak{p}^{\mathrm{long}}$, 
let $Q_w\in \mathcal{O}_0$ denote the quotient of $P_w$ modulo 
the trace in $P_w$ of all $P_y$, where $y\in W$ is such that 
$y<w$ with respect to the Bruhat order and $y\neq xw$, for any
$x\in W^{\mathfrak{p}}$. Denote by ${\Delta}^\mathfrak{p}_w$
the object $\pi_\mathfrak{p}(Q_w)$ in $\mathcal{S}_0^\mathfrak{p}$. 
The object ${\Delta}^\mathfrak{p}_w$ is called the
{\em standard object} corresponding to $w$.

The object ${\Delta}^\mathfrak{p}_w$ has a filtration with subquotients
$\overline{\Delta}^\mathfrak{p}_w$ (up to graded shift). The length of 
this filtration is $|W^{\mathfrak{p}}|$. With more details for the 
graded version: for $i\in\mathbb{Z}$, the multiplicity of 
$\overline{\Delta}^\mathfrak{p}_w\langle -2i\rangle$ as a subquotient of a
(graded) proper standard filtration of ${\Delta}^\mathfrak{p}_w$ equals
the cardinality of the set $\{w\in W^\mathfrak{p}\,:\,\ell(w)=i\}$.
Furthermore, each projective object in 
$\mathcal{S}_0^\mathfrak{p}$ has a filtration with standard subquotients. 

The simple preserving duality $\star$ on $\mathcal{O}_0$
induces a simple preserving duality on $\mathcal{S}_0^\mathfrak{p}$ which we
will denote by the same symbol, see \cite[Lemma~2.12]{MS}.

The above means that the underlying algebra $A^\mathfrak{p}$ 
of the category $\mathcal{S}_0^\mathfrak{p}$ is properly stratified in the
sense of \cite{Dl}. The objects $T^\mathfrak{p}_w$ are tilting 
with respect to this structure, in the sense of \cite{AHLU}. 
An additional property of the algebra $A^\mathfrak{p}$ is that 
each $T^\mathfrak{p}_w$ is also cotilting. This follows from the description
of tilting modules for $A^\mathfrak{p}$ in \cite[Section~6]{FKM} and the
fact that these modules are self-dual.

\section{First extension from a simple to a proper
standard module in $\mathcal{S}^{\mathfrak{p}}$}\label{s9}

\subsection{First extension from the antidominant simple}\label{s9.1}

Similarly as in $\cO$, it is easy to separately treat the following special case.

\begin{theorem}\label{fromw0inS}
For $y\in  \mathtt{X}_\mathfrak{p}^{\mathrm{long}}$ and $i\in\mathbb Z$, we have
\begin{displaymath}
\Ext^1_\mathcal{S}(L^\mathfrak{p}_{w_0},\overline{\Delta}^{\mathfrak p}_y) = 
\dim \ext^1_\mathcal{S}\big(L^\mathfrak{p}_{w_0}\langle -\ell(w_0)+2\rangle,\overline{\Delta}^{\mathfrak p}_y\langle -\ell(y)\rangle\big) =
\mathbf{c}(w_0w_0^{\mathfrak p}y). 
\end{displaymath}
\end{theorem}

\begin{proof}
Note that $L^\mathfrak{p}_{w_0}\cong \overline{\Delta}^\mathfrak{p}_{w_0}$
is a proper standard object. 
The object $T^\mathfrak{p}_{y}$ is both tilting and cotilting. In particular, it is the cotilting envelope of 
$\overline{\Delta}^\mathfrak{p}_{y}$. Let $Q$ be such that the following sequence is short exact in $(\mathcal{S}_0^\mathfrak{p})^\mathbb{Z}$:
\begin{equation}\label{eq3}
0\to  \overline{\Delta}^\mathfrak{p}_{y}\langle -\ell(y) \rangle
\to T^\mathfrak{p}_{y}\langle 2\ell(w_0^\mathfrak{p})-\ell(y) \rangle \to Q \to 0.
\end{equation}
Set $a:=2\ell(w_0^\mathfrak{p})-\ell(y)$.
As proper standard and costandard (and hence also cotiltitng)
objects are homologically orthogonal, it follows that 
\begin{displaymath}
\dim\mathrm{ext}^1_{\mathcal{S}}(L^{\mathfrak{p}}_{w_0}\langle i \rangle,
\overline{\Delta}^{\mathfrak{p}}_y\langle -\ell(y) \rangle)=
\dim\mathrm{hom}_{\mathcal{S}}(L^{\mathfrak{p}}_{w_0}\langle i \rangle,Q)
- \dim\mathrm{hom}_{\mathcal{S}}(L^{\mathfrak{p}}_{w_0}\langle i \rangle,T^\mathfrak{p}_{y}\langle a \rangle)+1.
\end{displaymath}
At the same time, we have
$w_0^\mathfrak{p}y\in \mathtt{X}_\mathfrak{p}^{\mathrm{short}}$.
Therefore,
$\overline{\Delta}^{\mathfrak{p}}_y\cong 
\pi_\mathfrak{p}(\Delta_{w_0^\mathfrak{p}y}\langle \ell(w_0^{\mathfrak p})\rangle)$
and $T^\mathfrak{p}_{y}\cong\pi_\mathfrak{p}(T_{w_0^\mathfrak{p}y}
\langle -\ell(w_0^{\mathfrak p})\rangle)$.
It follows that the sequence given by Formula~\ref{eq3} is obtained
by applying $\pi_\mathfrak{p}$ to the following short
exact sequence in $\mathcal{O}$:
\begin{displaymath}
0\to \Delta_{w_0^\mathfrak{p}y}\langle -\ell(w_0^{\mathfrak p}y) \rangle\to T_{w_0^\mathfrak{p}y} \langle -\ell(w_0^\mathfrak{p}y) \rangle \to Q'\to 0. 
\end{displaymath}
Since $Q'$ has a Verma flag, the socle of $Q'$ is a direct sum of copies
of shifts of $\Delta_{w_0}$, and $w_0\in \mathtt{X}_\mathfrak{p}^{\mathrm{long}}$. 
Consequently, $\pi_\mathfrak{p}$ induces
isomorphisms
\[\mathrm{hom}_{\mathcal{S}}(L^{\mathfrak{p}}_{w_0}\langle i\rangle,Q)=
\mathrm{hom}_{\mathcal{O}}(L_{w_0}\langle i\rangle,Q')\]
and
\[
\mathrm{hom}_{\mathcal{S}}(L^{\mathfrak{p}}_{w_0}
\langle i\rangle,T^\mathfrak{p}_{y}\langle-\ell(w_0^\mathfrak{p}y)  \rangle)=
\mathrm{hom}_{\mathcal{O}}(L_{w_0}\langle i\rangle,T_{w_0^\mathfrak{p}y}
\langle a\rangle).
\]

This implies that
\begin{displaymath}
\mathrm{ext}^1_{\mathcal{S}}(L^{\mathfrak{p}}_{w_0}\langle i\rangle,
\overline{\Delta}^{\mathfrak{p}}_y\langle -\ell(y) \rangle)=
\mathrm{ext}^1_{\mathcal{O}}(L_{w_0}\langle i\rangle,{\Delta}_{w_0^\mathfrak{p}y}\langle-\ell(w_0^\mathfrak{p}y)  \rangle)
\end{displaymath}
and the claim of the theorem now follows from Theorem~\ref{thm1}. 
\end{proof}

\subsection{Inclusions between proper standard modules}\label{s4.05}

Recall from Subsection~\ref{ss:socandext} the properties of 
homomorphisms between Verma modules in $\cO$. Applying the functor 
$\pi_\mathfrak{p}$ gives:
\begin{itemize}
\item every non-zero map between two proper 
standard objects in $\mathcal{S}_0^\mathfrak{p}$ is injective;
\item for $x,y\in \mathtt{X}_\mathfrak{p}^{\mathrm{long}}$, 
we have $\mathrm{hom}_{\mathcal{S}_0}(\overline{\Delta}^\mathfrak{p}_x,
\overline{\Delta}^\mathfrak{p}_y\langle d \rangle)\neq 0$ 
if and only if $x\geq y$ and $d = \ell(y)-\ell(x)$;
\item $\dim \mathrm{Hom}_{\mathcal{S}_0}(\overline{\Delta}^\mathfrak{p}_x,
\overline{\Delta}^\mathfrak{p}_y)\leq 1$, for all 
$x,y\in \mathtt{X}_\mathfrak{p}^{\mathrm{{long}}}$.
\end{itemize}
We thus obtain the canonical quotients $\overline{\Delta}^\mathfrak{p}_y/
\overline{\Delta}^\mathfrak{p}_x:=\overline{\Delta}^\mathfrak{p}_y/(
\overline{\Delta}^\mathfrak{p}_x\langle \ell(y)-\ell(x) \rangle)$.
The following analogue of Proposition~\ref{extandsoc} relates these 
quotients to extensions from simple to proper standard objects in 
$\mathcal{S}_0^\mathfrak{p}$.

\begin{proposition}\label{extandsocinS}
For each $x,y\in  \mathtt{X}_\mathfrak{p}^{\mathrm{long}}$ with $x\neq w_0$, 
we have
\[\dim\ext^1(L^\mathfrak{p}_x\langle d \rangle,
\overline{\Delta}^\mathfrak{p}_y\langle -\ell(y)\rangle)=
[\mathrm{soc}\,
\overline{\Delta}^\mathfrak{p}_e/\overline{\Delta}^\mathfrak{p}_y:
L^\mathfrak{p}_x\langle d \rangle].\]
\end{proposition}

\begin{proof}
Let $L:= (L^{\mathfrak{p}}_x \langle d \rangle ) ^{\oplus m}$ and suppose we have a short exact sequence
\begin{equation}\label{eq2}
0\to  \overline{\Delta}^{\mathfrak{p}}_y\langle -\ell(y) \rangle\to 
M\to L \to 0
\end{equation}
such that $M$ is indecomposable. Since $L$
is semisimple, we have 
\[\soc M = \soc\overline{\Delta}^{\mathfrak{p}}_y\langle -\ell(y) \rangle = L^{\mathfrak{p}}_{w_0}\langle -\ell(w_0) \rangle.\] 
Thus, the injective covers of $\overline{\Delta}^{\mathfrak{p}}_y$
and of $M$ coincide and are isomorphic to 
$I^{\mathfrak{p}}_{w_0}\langle -\ell(w_0) \rangle$. The latter is also isomorphic to a shift of $P^{\mathfrak{p}}_{w_0}$.

Being both a tilting and a cotilting object,  $P^{\mathfrak{p}}_{w_0}$ has a
proper standard filtration which starts with a submodule isomorphic to
$\overline{\Delta}^{\mathfrak{p}}_{w_0^\mathfrak{p}}$, up to shift.
In particular, the cokernel of the inclusion
\begin{displaymath}
0\to \overline{\Delta}^{\mathfrak{p}}_{w_0^\mathfrak{p}}\langle -\ell(w_0^\mathfrak{p}) \rangle
\to I_{w_0}^\mathfrak{p}\langle -\ell(w_0) \rangle
\end{displaymath}
has a proper standard filtration. 

As the socle of each proper standard module is a 
shift of $L_{w_0}^\mathfrak{p}$, we have 
\begin{displaymath}
\mathrm{Ext}^1_{\mathcal{S}}(L^{\mathfrak{p}}_w,
\overline{\Delta}^{\mathfrak{p}}_{w_0^\mathfrak{p}})=0,
\end{displaymath}
for any $w\in \mathtt{X}_\mathfrak{p}^{\mathrm{long}}$
such that $w\neq w_0$. This implies that $M$ must be a submodule of 
$\overline{\Delta}^{\mathfrak{p}}_{w_0^\mathfrak{p}}
\langle -\ell(w_0^\mathfrak{p}) \rangle$.
In other words, $L$ should be a summand of the
socle of the cokernel of the canonical inclusion
$\overline{\Delta}^{\mathfrak{p}}_{y}\langle -\ell(y) \rangle\subset 
\overline{\Delta}^{\mathfrak{p}}_{w_0^\mathfrak{p}}\langle -\ell(w_0^\mathfrak{p}) \rangle$. 

On the other hand, any summand of this socle gives rise
to a non-split short exact sequence as in Formula~\eqref{eq2} 
(since in that case $M$ obviously has simple socle). The claim follows.
\end{proof}

\subsection{Cokernel of inclusion of proper standard modules}\label{s9.3}

\begin{lemma}\label{lem4-11}
Let $x,y\in \mathtt{X}_\mathfrak{p}^{\mathrm{short}}$ be such that
$x\geq y$. Let $z\in W$ be such that $L_z$ appears in the socle
of $\Delta_y/\Delta_x$. Then $z\in \mathtt{X}_\mathfrak{p}^{\mathrm{long}}$.
\end{lemma}

\begin{proof}
Note that $x\in \mathtt{X}_\mathfrak{p}^{\mathrm{short}}$ is equivalent to $sx>x$, for each $s\in S\cap W^{\mathfrak p}$. 
Thus, if $L_z$ appears in the socle of $\Delta_y/\Delta_x$, (and thus in the socle of $\Delta_e/\Delta_x$,) then, 
by \cite[Proposition~6]{KMM3}, we have $sz<z$, for each $s\in S\cap W^{\mathfrak p}$. 
The latter is equivalent to $z\in \mathtt{X}_\mathfrak{p}^{\mathrm{long}}$, as desired.
\end{proof}

\begin{proposition}\label{prop4-12'}
For $x,y\in \mathtt{X}_\mathfrak{p}^{\mathrm{long}}$ 
such that  $x\geq y$, we have
\[\soc (\overline{\Delta}^\mathfrak{p}_y / \overline{\Delta}^\mathfrak{p}_x) 
\cong \pi_\mathfrak{p}(\soc\Delta_{w_0^\mathfrak{p}y}/\Delta_{w_0^\mathfrak{p}x}).\]
This isomorphism holds as well for graded modules with the standard shifts, 
that is, if we shift each $\overline{\Delta}^\mathfrak{p}_w$ or $\Delta_w$ 
by $\langle -\ell(w)\rangle$.
\end{proposition}

\begin{proof}
As mentioned in Subsection~\ref{s4.3}, we have the isomorphisms
\begin{displaymath}
\overline{\Delta}^\mathfrak{p}_x\cong
\pi_{\pi}(\Delta_{w_0^\mathfrak{p}x})\langle\ell(w_0^\mathfrak{p})\rangle
\quad\text{ and }\quad
\overline{\Delta}^\mathfrak{p}_y\cong
\pi_{\pi}(\Delta_{w_0^\mathfrak{p}y})\langle\ell(w_0^\mathfrak{p})\rangle.
\end{displaymath}
Note that 
$w_0^\mathfrak{p}x,w_0^\mathfrak{p}y\in \mathtt{X}_\mathfrak{p}^{\mathrm{short}}$. 
Therefore
we may apply Lemma~\ref{lem4-11} to conclude that the socle of 
the cokernel of the inclusion 
$\Delta_{w_0^\mathfrak{p}x}\subset \Delta_{w_0^\mathfrak{p}y}$
contains only $L_z$ such that $z\in \mathtt{X}_\mathfrak{p}^{\mathrm{long}}$.
Now the claim of the proposition follows by applying $\pi_\mathfrak{p}$. 
\end{proof}

\subsection{Ungraded statements in type $A$}\label{s9.4}

In type $A$, the above results can be summarized
and made more precise as follows.

\begin{proposition}\label{prop4-12}
In type $A$, for $x,y\in \mathtt{X}_\mathfrak{p}^{\mathrm{long}}$ such 
that  $x\geq y$, the cokernel of $\overline{\Delta}^\mathfrak{p}_x\subset
\overline{\Delta}^\mathfrak{p}_y$ is isomorphic
to the (multiplicity-free) direct sum of all simples $L_z^{\mathfrak{p}}$, where 
$z\in \mathtt{X}_\mathfrak{p}^{\mathrm{long}}$ and
$z\in \Phi(\mathbf{BM}(x)\setminus\mathbf{BM}(y))$.
\end{proposition}
\begin{proof}
This follows from Corollary~\ref{corthm2} and Proposition~\ref{prop4-12'}. 
\end{proof}

\begin{theorem}\label{thm4-21}
In type $A$, let $x,y\in \mathtt{X}_\mathfrak{p}^{\mathrm{long}}$. Then we have
\begin{equation}\label{eq1}
\dim\mathrm{Ext}^1_{\mathcal{S}}(L^{\mathfrak{p}}_x,
\overline{\Delta}^{\mathfrak{p}}_y)=
\begin{cases}
\mathbf{c}(w_0w_0^\mathfrak{p}y),& x=w_0;\\
1,& x\in\Phi(\mathbf{BM}(y));\\
0,&  \text{otherwise}.
\end{cases}
\end{equation}
\end{theorem}

\begin{proof}
The case $x=w_0$ is covered by Theorem~\ref{fromw0inS}.
For $x\neq w_0$, Formula~\eqref{eq1} 
follows from Proposition~\ref{prop4-12} and Proposition~\ref{extandsocinS}.
\end{proof}

\subsection{Graded statement in type $A$}\label{s9.5}

We can also explicitly determine the degree shifts for 
the graded non-zero extensions in Theorem~\ref{thm4-21}.

\begin{proposition}\label{prop-4-31}
Assume we are in type $A$. Let $y\in \mathtt{X}_\mathfrak{p}^{\mathrm{long}}$
and $x=\Phi((i,j,k))$, for some 
$(i,j,k)\in \mathbf{BM}(y)\cap
\mathtt{X}_\mathfrak{p}^{\mathrm{long}}$.
Then the unique degree $m\in\mathbb{Z}$ for which
$\dim\mathrm{ext}_{}^{1}(L^\mathfrak{p}_x\langle -m\rangle,\overline\Delta^{\mathfrak{p}}_y\langle -\ell(y)\rangle)=1$ is 
\[m = \frac{(n-1)(n-2)}{2}+\vert i-j\vert+2k.\]
\end{proposition}

\begin{proof}
This follows from Proposition~\ref{prop4-12'}, Proposition~\ref{extandsocinS} and Proposition~\ref{cor2graded}.
\end{proof}

\subsection{First extension from other simples to proper standard modules in other types}
\label{9.6}

Proposition~\ref{extandsocinS} and Proposition~\ref{prop4-12'} translate all graded and ungraded results from \cite{KMM2} to the corresponding statements on the first extension spaces.
In particular, we have
\begin{itemize}
    \item for $x,y\in \mathtt{X}_\mathfrak{p}^{\mathrm{long}}$, we have
 $\Ext^1(L^\mathfrak{p}_x,\overline\Delta^\mathfrak{p}_w) = 0 $ unless $x\in \mathcal J$;
 \item if $x\in \mathcal J$, we have
    \[\dim\Ext^1(L^\mathfrak{p}_x,\overline\Delta_w^\mathfrak{p})\leq |_s\mathbf{BM}_t(w) |,\]
     where  $s,t$ are simple reflections in $W$ such that $sx>x$ and $xt>x$.
\end{itemize}

We emphasize that the main point of giving the above bound is to have a general 
statement, and that the bound $|_s\mathbf{BM}_t(w) |$ is a gross exaggeration
in most of the cases. For computing/bounding first extension spaces between 
simple and proper standard modules, it is strongly recommended to ignore 
the bound $|_s\mathbf{BM}_t(w) |$ and instead look at  \cite[Section~5]{KMM2} 
(see also the discussion after \cite[Theorem~F]{KMM2}).

\section{First extension from a simple to a 
standard module in $\mathcal{S}^{\mathfrak{p}}$}\label{s6}

\subsection{Elementary general observations}\label{s6.1}

Since $w_0$ corresponds to the minimum element for the partial order
with respect to which $A^\mathfrak{p}$ is stratified,
the standard object $\Delta^\mathfrak{p}_{w_0}$ is a tilting object. 
Due to the special properties of $A^\mathfrak{p}$ mentioned at the 
end of Subsection~\ref{s4.3}, it is also a cotilting module.
The simple object $L^\mathfrak{p}_{w_0}$ is a proper standard module.
Therefore, due to the homological orthogonality of 
proper standard and cotilting modules, we have
\begin{displaymath}
\mathrm{Ext}^i_{\mathcal{S}}(L^\mathfrak{p}_{w_0},
\Delta^\mathfrak{p}_{w_0})=0,\quad\text{ for all }i>0.
\end{displaymath}

The projective-injective object $I_{w_0}^\mathfrak{p}$ is a tilting object
and is thus the tilting envelope of the standard object 
$\Delta^\mathfrak{p}_{w_0^\mathfrak{p}}$. Therefore the cokernel of 
the inclusion $\Delta^\mathfrak{p}_{w_0^\mathfrak{p}}\hookrightarrow
I_{w_0}^\mathfrak{p}$ has a standard filtration. As the socle of 
each standard object is isomorphic to $L^\mathfrak{p}_{w_0}$, it follows
that the only simple object appearing in the socle of the cokernel of
the above inclusion is $L^\mathfrak{p}_{w_0}$. Consequently,
\begin{displaymath}
\mathrm{Ext}^1_{\mathcal{S}}(L^\mathfrak{p}_{x},
\Delta^\mathfrak{p}_{w_0^\mathfrak{p}})=0,\quad\text{ for all }
x\in \mathtt{X}_\mathfrak{p}^{\mathrm{long}}\setminus\{w_0\}.
\end{displaymath}
We will generalize this result below in Subsection~\ref{s6.2}.

\subsection{Reduction to category $\mathcal{O}$}\label{s6.3}

The following statement reduces the problem of computing 
first extensions from simple to standard objects in $\mathcal{S}$
to the problem of computing first extensions between certain
modules in $\mathcal{O}$.

\begin{proposition}\label{prop65}
For $x,y\in \mathtt{X}_\mathfrak{p}^{\mathrm{long}}$
and $i\in\mathbb{Z}$,
we have an isomorphism
\begin{displaymath}
\mathrm{ext}^1_{\mathcal{S}}
(L^\mathfrak{p}_x,\Delta^\mathfrak{p}_y\langle i\rangle)\cong
\mathrm{ext}^1(L_x,Q_y\langle i\rangle).
\end{displaymath}
\end{proposition}

\begin{proof}
The functor $\pi_\mathfrak{p}$ connects $\mathcal{O}_0$ 
and $\mathcal{S}^\mathfrak{p}_0$. Since 
$\pi_\mathfrak{p}(L_x)=L^\mathfrak{p}_x$
and $\pi_\mathfrak{p}(Q_y)=\Delta^\mathfrak{p}_y$, we need to show
that the socle of the cokernel $C_y$ of the natural inclusion $Q_y\hookrightarrow I_{w_0}$
only contains simples of the form $L_z$, where 
$z\in \mathtt{X}_\mathfrak{p}^{\mathrm{long}}$.

Let $\mathfrak{a}$ be the semi-simple part of $\mathfrak{p}$. 
For $w\in \mathtt{X}_\mathfrak{p}^{\mathrm{long}}$, the module
$Q_w$ is obtained by parabolic induction (from $\mathfrak{p}$ to 
$\mathfrak{g}$) of a projective-injective module in the category
$\mathcal{O}$ for $\mathfrak{a}$, see \cite[Proposition~2.9]{MS}.
In particular, $Q_w$ is an (infinite)
direct sum of projective-injective modules in the category
$\mathcal{O}$ for $\mathfrak{a}$. 
Since $I_{w_0}=P_{w_0}$ has a filtration 
whose subquotients are various $Q_w$'s, the module $I_{w_0}$
is an (infinite) direct sum of projective injective module in the category
$\mathcal{O}$ for $\mathfrak{a}$.  Consequently,
 the module $C_y$ 
is an (infinite) direct sum of projective injective module in the category
$\mathcal{O}$ for $\mathfrak{a}$. In particular, for any simple root
$\alpha$ of $\mathfrak{a}$, the action of a non-zero element in
$\mathfrak{a}_{\-\alpha}$ on any simple submodule $L_z$ of $C_y$
is injective.

This means that $sz>z$, for any simple reflection $s\in W^\mathfrak{p}$,
and hence $z\in \mathtt{X}_\mathfrak{p}^{\mathrm{long}}$ as asserted.
Now the statement of the proposition follows by comparing the long exact
sequence obtained by applying $\mathrm{hom}^1(L_x,{}_-\langle i\rangle)$
to the short exact sequence 
\begin{displaymath}
0\to  Q_y\to I_{w_0}\to C_y\to 0
\end{displaymath}
with the image of this long exact sequence under $\pi_\mathfrak{p}$. 
\end{proof}

As a corollary, we have the following general observation:

\begin{corollary}\label{cor67}
For $x,y\in \mathtt{X}_\mathfrak{p}^{\mathrm{long}}$
and $i\in\mathbb{Z}$. If
$\mathrm{ext}^1_{\mathcal{S}}
(L^\mathfrak{p}_x,\Delta^\mathfrak{p}_y\langle i\rangle)\neq 0$,
then $x\in \mathcal{J}\cup\{w_0\}$.
\end{corollary}

\begin{proof}
By Proposition~\ref{prop65}, we need to show that the assumption
$\mathrm{ext}^1(L_x,Q_y\langle i\rangle)\neq 0$
implies $x\in \mathcal{J}\cup\{w_0\}$. The module $Q_w$
has a Verma flag, by construction. From \cite[Proposition~3]{KMM3}
it follows that
$\mathrm{ext}^1(L_x,\Delta_w\langle i\rangle)\neq 0$, for $w\in W$,
implies $x\in \mathcal{J}\cup\{w_0\}$. As any non-zero extension 
from $L_x$ to $Q_y\langle i\rangle$ must induce a non-zero
extension from $L_x$ to one of the Verma subquotients of $Q_y\langle i\rangle$,
the claim of the corollary follows.
\end{proof}

\subsection{The case of standard modules which can be obtained
using projective functors}\label{s6.2}

An element $w\in W$ is called {\em ($\mathfrak p$-)special} provided that 
the subgroup $w\inv W^\mathfrak{p}w$ is parabolic, that is, there exists a parabolic subgroup 
${W}^\mathfrak{\tilde p}$ of $W$ such that 
$W^\mathfrak{p}w=w{W}^\mathfrak{\tilde p}$. 
For example, any $w\in W^{\mathfrak p}$, in particular $w_0^\mathfrak{p}$, is special.
Also, $w_0$ is special, for we can choose 
${W}^\mathfrak{\tilde p}=w_0{W}^\mathfrak{p}w_0$.

\begin{proposition}\label{prop-singular}
Let $x,y\in\mathtt{X}_\mathfrak{p}^{\mathrm{long}}$ and assume that $y$ is special. 
\begin{enumerate}[$($i$)$]
\item \label{prop-singular.1} We have
\[\Ext^1_{\mathcal S}(L_x^\mathfrak{p},\Delta_y^\mathfrak{p})\cong \Ext^1_{\cO}(L(x\cdot \lambda),\Delta(y\cdot\lambda)),\]
where $\lambda$ is an integral dominant weight which has the dot-stabilizer $W^\mathfrak{\tilde p}$. 
\item \label{prop-singular.2} Under the additional assumption $x\neq w_0$, we have
\[\dim \Ext^1_{\mathcal S}(L_x^\mathfrak{p},\Delta_y^\mathfrak{p}) = \dim \Ext^1_{\cO}(L(x\cdot \lambda),\Delta(y\cdot\lambda)) = 
[\soc\co_{w_0^\mathfrak{p}y} : L_x].\]
\end{enumerate}
\end{proposition}

\begin{proof}
Let $x,y$ be as above and let $\mathfrak{\tilde p}$ be such that $W^{\mathfrak p}y=y{W}^\mathfrak{\tilde p}$. Let $\tilde{w}_0$ be the longest element in ${W}^\mathfrak{\tilde p}$.
We have
$Q_w\cong \theta_{\tilde{w}_0}\Delta_{w}$, since both sides are characterized as the quotient of $P_w$ with a filtration where the factors are exactly $\Delta_{z}\langle -\ell(w)+\ell(z) \rangle$ for $z\in W^\mathfrak{p}w=w{W}^\mathfrak{\tilde p}$ (with multiplicity one). 
Let $\lambda$ be a dominant integral
weight for which ${W}^\mathfrak{\tilde p}$ is the dot-stabilizer.
Let $\mathcal{O}_\lambda$ be the corresponding block of $\mathcal{O}$.
Consider the corresponding projective functors
\begin{displaymath}
\theta_{\tilde{w}_0}^\mathrm{on}:\mathcal{O}_0\to\mathcal{O}_\lambda\quad
\text{ and }\quad
\theta_{\tilde{w}_0}^\mathrm{out}:\mathcal{O}_\lambda\to\mathcal{O}_0
\end{displaymath}
of translation onto and out of the ${W}^\mathfrak{\tilde p}$-wall, respectively.
These functors are biadjoint and 
$\theta_{\tilde{w}_0}\cong 
\theta_{\tilde{w}_0}^\mathrm{out}\theta_{\tilde{w}_0}^\mathrm{on}$. 
In particular, for  $x\in \mathtt{X}_\mathfrak{p}^{\mathrm{long}}$, we have
\begin{displaymath}
\mathrm{Ext}^1(L_x,Q_w)\cong
\mathrm{Ext}^1(\theta_{\tilde{w}_0}^\mathrm{on}L_x,
\theta_{\tilde{w}_0}^\mathrm{on}\Delta_w).
\end{displaymath}
Since $x\in \mathtt{X}_\mathfrak{p}^{\mathrm{long}}$, we have $\theta_{\tilde{w}_0}^\mathrm{on}L_x\cong L(x\cdot \lambda)$ in $\mathcal{O}_\lambda$. We also have
$\theta_{\tilde{w}_0}^\mathrm{on}\Delta_w\cong \Delta(w\cdot\lambda)$ in $\mathcal{O}_\lambda$. 
The claimed evaluation of $\dim \Ext^1_{\mathcal S}(L_x^\mathfrak{p},\Delta_y^\mathfrak{p})$ 
now follows from Proposition~\ref{prop65}.  

Now we prove the second equality in the second statement, where the first equality is obtained from the first claim. If $x\neq w_0$ 
then the proof of Proposition~\ref{extandsocinS} (or of Proposition~\ref{extandsoc}) identifies 
the value $\dim \Ext^1_{\cO}(L(x\cdot \lambda),\Delta(y\cdot\lambda))$ with the value
$[\soc\Delta(\lambda)/\Delta(y\cdot \lambda) : L(x\cdot\lambda)]$. The latter agrees with $[\soc\co_{w_0^{\mathfrak p}y} : L_x]$ 
by \cite[Proposition 15]{KMM3} since 
$w_0^{\mathfrak p}y=y\tilde{w}_0$ is the shortest element in $W^{\mathfrak p}y=y{W}^\mathfrak{\tilde p}$.
\end{proof}

\subsection{A type $A$ formula}

By Subsection~\ref{s3.35}, Proposition~\ref{prop-singular} completely computes the first extension between simple and standard in $\mathcal S$-subcategories in type $A$.

\begin{proposition}\label{prop-singularA}
Let $x,y\in\mathtt{X}_\mathfrak{p}^{\mathrm{long}}$ with $y$ special and 
assume we are in type $A$. Then 
\begin{displaymath}
\dim\mathrm{Ext}_{\mathcal S}^{1}(L_x^\mathfrak{p},\Delta_y^\mathfrak{p})=
\begin{cases}
\mathbf{c}(\overline{x}\underline{y})-\mathrm{rank}(W^\mathfrak{p}), & \overline{x}=w_0;\\
1, & \overline{x}\in \Phi(\mathbf{BM}(\underline{y}));\\
0, & \text{otherwise}. 
\end{cases}
\end{displaymath}
\end{proposition}

The graded version of this claim is obtain in the obvious way 
using the shifts described in Subsection~\ref{s3.4}.

\section{Examples}\label{s5}

\subsection{$\mathbf{sl}_3$-example}\label{s5.1}

Consider the case of the Lie algebra $\mathfrak{sl}_3$.
In this case we have $W=S_3=\{e,s,t,st,ts,w_0=sts=tst\}$.
Let $\mathfrak{p}$ be such that $W^\mathfrak{p}=\{e,s\}$. With such a choice, we have
\begin{displaymath}
\mathtt{X}_\mathfrak{p}^{\mathrm{long}}=\{s,st,w_0\}
\quad\text{ and }\quad
\mathtt{X}_\mathfrak{p}^{\mathrm{short}}=\{e,t,ts\}
\end{displaymath}
and the Hasse diagrams for the (opposite of the) Bruhat order on 
$W$, $\mathtt{X}_\mathfrak{p}^{\mathrm{long}}$ and 
$\mathtt{X}_\mathfrak{p}^{\mathrm{short}}$ are as follows:
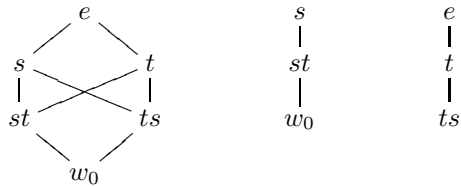
\begin{displaymath}
\xymatrix@R=3mm@C=3mm{
&e\ar@{-}[dr]\ar@{-}[dl]&\\
s\ar@{-}[drr]\ar@{-}[d]&&t\ar@{-}[dll]\ar@{-}[d]\\
st\ar@{-}[dr]&&ts\ar@{-}[dl]\\
&w_0&
}\qquad\qquad
\xymatrix@R=3mm@C=3mm{s\ar@{-}[d]\\st\ar@{-}[d]\\w_0}\qquad\qquad
\xymatrix@R=3mm@C=3mm{e\ar@{-}[d]\\t\ar@{-}[d]\\ts}
\end{displaymath}

If we denote $L_x^\mathfrak{p}$ simply by $x$, then the 
subquotients of the graded filtrations of
the indecomposable projectives in $\mathcal{S}_0$ are as follows:
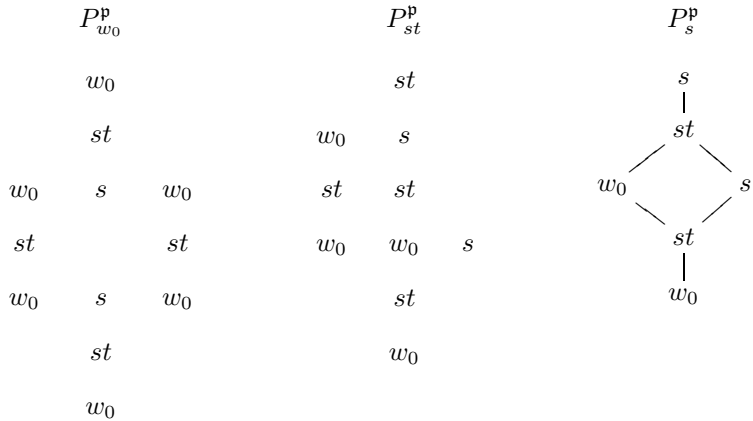
\begin{displaymath}
\xymatrix@R=3mm@C=3mm{&P_{w_0}^\mathfrak{p}&\\&w_0&\\&st&\\w_0&s&w_0\\
st&&st\\w_0&s&w_0\\&st&\\&w_0&}\qquad\qquad
\xymatrix@R=3mm@C=3mm{&P_{st}^\mathfrak{p}&\\&st&\\w_0&s&\\st&st&\\w_0&w_0&s\\&st&\\&w_0&}\qquad\qquad
\xymatrix@R=3mm@C=3mm{&P_s^\mathfrak{p}&\\&s\ar@{-}[d]&\\&st\ar@{-}[dl]\ar@{-}[dr]&
\\w_0\ar@{-}[dr]&&s\ar@{-}[dl]\\&st\ar@{-}[d]&\\&w_0&} 
\end{displaymath}
The (graded and unique) Loewy filtrations of the proper standard modules
are as follows:
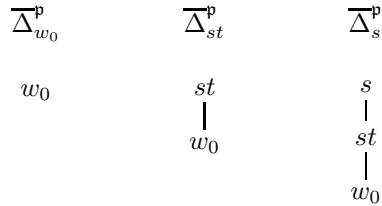
\begin{displaymath}
\xymatrix@R=3mm@C=3mm{\overline{\Delta}_{w_0}^\mathfrak{p}\\w_0}\qquad\qquad
\xymatrix@R=3mm@C=3mm{\overline{\Delta}_{st}^\mathfrak{p}\\st\ar@{-}[d]\\w_0}\qquad\qquad
\xymatrix@R=3mm@C=3mm{\overline{\Delta}_s^\mathfrak{p}\\s\ar@{-}[d]\\st\ar@{-}[d]\\w_0} 
\end{displaymath} 
We note that all proper standard modules are multiplicity-free and hence 
the corresponding module diagrams are well-defined. This is not the case 
for the indecomposable projectives $P_{w_0}^\mathfrak{p}$ and
$P_{st}^\mathfrak{p}$ which are not even graded multiplicity-free.
The projective $P_{s}^\mathfrak{p}$ is not multiplicity-free but it is
graded multiplicity-free and hence its module diagram is well-defined
as well as the algebra $A^\mathfrak{p}$ is positively graded.

The following table contains information on the {\color{magenta}dimension}
and the {\color{teal}degree shift} for the extension spaces from a simple
object to a proper standard object in the format $({\color{magenta}d},{\color{teal}m})$
for the formula ${\color{magenta}\dim}\,\mathrm{ext}^1_{\mathcal{S}_0}
(L_x^\mathfrak{p},\overline{\Delta}^\mathfrak{p}_y\langle{\color{teal}m}\rangle)$:
\begin{displaymath}
\begin{array}{c||c|c|c}
x\setminus y&s&st&w_0\\
\hline\hline
s&-&
({\color{magenta}1},{\color{teal}-1})&
-\\ \hline
st&-&
-&
({\color{magenta}1},{\color{teal}-1})\\ \hline
w_0&({\color{magenta}2},{\color{teal}0})&
({\color{magenta}2},{\color{teal}-1})&
({\color{magenta}1},{\color{teal}-2})
\end{array}
\end{displaymath}

Note that $s$ and $w_0$ are special while $st$ is not.
The following table contains information on the {\color{magenta}dimension}
and the {\color{teal}degree shift} for the extension spaces from a simple
object to a standard object in the format $({\color{magenta}d},{\color{teal}m})$
for the formula ${\color{magenta}\dim}\,\mathrm{ext}^1_{\mathcal{S}_0}
(L_x^\mathfrak{p},{\Delta}^\mathfrak{p}_y\langle{\color{teal}m}\rangle)$:
\begin{displaymath}
\begin{array}{c||c|c|c}
x\setminus y&s&st&w_0\\
\hline\hline
s&-&
({\color{magenta}1},{\color{teal}1})&
-\\ \hline
st&-&
-&
({\color{magenta}1},{\color{teal}1})\\ \hline
w_0&({\color{magenta}1},{\color{teal}2})&
({\color{magenta}1},{\color{teal}1})&
-
\end{array}
\end{displaymath}

\subsection{$\mathbf{sl}_4$-example}\label{s5.2}

The Lie algebra $\mathbf{sl}_4$ is the smallest Lie algebra for which 
there are non-trivial Kazhdan-Lusztig polynomials. These non-trivial
KL-polynomials also contribute to a non-trivial extension from a
simple module to a Verma module. 

We have $W=S_4$ and let $s_1$, $s_2$ and $s_3$ be the simple 
reflections with the corresponding Dynkin diagram
\begin{displaymath}
\xymatrix{s_1\ar@{-}[rr]&&s_2\ar@{-}[rr]&&s_3}.
\end{displaymath}
As pointed out in \cite[Subsection~1.3]{KMM3}, we have 
the following fact (which we present here in the graded version):
\begin{displaymath}
\mathrm{ext}^1(L_{s_2w_0}\langle-3\rangle,\Delta_{s_2}\langle-1\rangle)\cong\mathbb{C}. 
\end{displaymath}
Note that $s_2w_0$ is a longest representative in the cosets
${}_{W^\mathfrak{p}}\setminus W$ for the choices of a parabolic 
subgroups $W^\mathfrak{p}$ in $W$ given by the following subsets of simple roots:
\begin{displaymath}
\varnothing,\quad  \{s_1\},\quad \{s_2\}, \quad \{s_1,s_2\}. 
\end{displaymath}
We denote the corresponding parabolic subalgebras 
by $\mathfrak{p}_i$, for $i=1,2,3,4$. Consequently, we have:
\begin{gather*}
\mathrm{ext}^1_{\mathcal{S}^\mathfrak{p}_1}
(L_{s_2w_0}^{\mathfrak{p}_1}\langle-3\rangle,
\overline{\Delta}_{s_2}^{\mathfrak{p}_1}\langle-1\rangle)\cong\mathbb{C},\\
\mathrm{ext}^1_{\mathcal{S}^\mathfrak{p}_2}
(L_{s_2w_0}^{\mathfrak{p}_2}\langle-3\rangle,
\overline{\Delta}_{s_1s_2}^{\mathfrak{p}_2}\langle-2\rangle)\cong\mathbb{C},\\
\mathrm{ext}^1_{\mathcal{S}^\mathfrak{p}_3}
(L_{s_2w_0}^{\mathfrak{p}_3}\langle-3\rangle,
\overline{\Delta}_{s_3s_2}^{\mathfrak{p}_3}\langle-2\rangle)\cong\mathbb{C},\\
\mathrm{ext}^1_{\mathcal{S}^\mathfrak{p}_4}
(L_{s_2w_0}^{\mathfrak{p}_4}\langle-3\rangle,
\overline{\Delta}_{s_1s_3s_2}^{\mathfrak{p}_4}\langle-3\rangle)\cong\mathbb{C}. 
\end{gather*}

\vspace{2mm}

\noindent
H.~K.: Department of Mathematics, Uppsala University, Box. 480,
SE-75106, Uppsala,\\ SWEDEN, email: {\tt hankyung.ko\symbol{64}math.uu.se}

\noindent
V.~M.: Department of Mathematics, Uppsala University, Box. 480,
SE-75106, Uppsala,\\ SWEDEN, email: {\tt mazor\symbol{64}math.uu.se}

\end{document}